\documentclass[a4paper,10pt]{amsart}

\usepackage[utf8x]{inputenc}
\usepackage{amsfonts}
\usepackage{amsmath}
\usepackage{amssymb}
\usepackage[all]{xy}
\usepackage{mathtools}

\usepackage[draft=false]{hyperref}  

\DeclareMathOperator{\id}{id}
\DeclareMathOperator{\Hom}{Hom}

\DeclareMathOperator{\HF}{\mathbf{H}\mathfrak{F}}

\DeclareMathOperator{\vcd}{vcd}
\DeclareMathOperator{\cd}{cd}
\DeclareMathOperator{\pd}{pd}

\DeclareMathOperator{\HeckeFcd}{\mathcal{H}_\mathfrak{F}cd}

\DeclareMathOperator{\FP}{FP}

\DeclareMathOperator{\Ext}{Ext}

\DeclareMathOperator{\ZZ}{\mathbb{Z}}
\DeclareMathOperator{\QQ}{\mathbb{Q}}
\DeclareMathOperator{\ucd}{\underline{cd}}
\DeclareMathOperator{\ugd}{\underline{gd}}
\newcommand{\uE}{\underline{\mathrm{E}}}

\usepackage{amsthm}
\theoremstyle{plain}
\newtheorem{Lemma}{Lemma}[section]
\newtheorem{Theorem}[Lemma]{Theorem}
\newtheorem{Cor}[Lemma]{Corollary}
\newtheorem{Prop}[Lemma]{Proposition}

\theoremstyle{definition}

\newtheorem{Example}[Lemma]{Example}
\newtheorem{Question}[Lemma]{Question}
\newtheorem{Remark}[Lemma]{Remark}
\newtheorem*{Acknowledgements*}{Acknowledgements}

\DeclareMathOperator{\silp}{silp}
\DeclareMathOperator{\spli}{spli}
\DeclareMathOperator{\Gcd}{Gcd}
\DeclareMathOperator{\Gpd}{Gpd}
\DeclareMathOperator{\Fpd}{\mathfrak{F}pd}
\DeclareMathOperator{\FGcd}{\mathfrak{F}_{\mathrm{G}}cd}
\DeclareMathOperator{\FExt}{\mathfrak{F}Ext}
\DeclareMathOperator{\GExt}{GExt}

\DeclareMathOperator{\FH}{\mathfrak{F}\mathit{H}}
\DeclareMathOperator{\FhatH}{\widehat{\mathfrak{F}\mathit{H}}}
\DeclareMathOperator{\GH}{G\mathit{H}}
\DeclareMathOperator{\FGExt}{\mathfrak{F}_GExt}
\DeclareMathOperator{\Fcd}{\mathfrak{F}cd}
\DeclareMathOperator{\FGH}{\mathfrak{F}_G\mathit{H}}
\DeclareMathOperator{\FGpd}{\mathfrak{F}_Gpd}

\title{On the Gorenstein and \texorpdfstring{$\mathfrak{F}$}{F}-cohomological dimensions}
\author{Simon St. John-Green}
\email{Simon.StJG@gmail.com}
\address{Department of Mathematics, University of Southampton, SO17 1BJ, UK}
\date{\today}
\subjclass{20J05, 18G}

\keywords{$\mathfrak{F}$-Cohomological Dimension, Gorenstein Cohomological Dimension}

\begin{document}

{\abstract{We prove that for any discrete group $G$ with finite $\mathfrak{F}$-cohomological dimension, the Gorenstein cohomological dimension equals the $\mathfrak{F}$-cohomological dimension.  This is achieved by constructing a long exact sequence of cohomological functors, analogous to that constructed by Avramov and Martsinkovsky in \cite{AvramovMartsinkovsky-AbsoluteRelativeAndTateCohomology-FiniteGorenstein}, containing the $\mathfrak{F}$-cohomology and complete $\mathfrak{F}$-cohomology.  As a corollary we improve upon a theorem of Degrijse concerning subadditivity of the $\mathfrak{F}$-cohomological dimension under group extensions \cite[Theorem B]{Degrijse-ProperActionsAndMackeyFunctors}.}}

\maketitle

\section{Introduction}\label{section:introduction}

Throughout, $G$ denotes a discrete group and $R$ a commutative ring.

Let $n_G$ denote the minimal dimension of a contractible proper $G$-CW-complex and $\ugd G$ the minimal dimension of a model for $\uE G$, the classifying space for proper actions of $G$.  Clearly $n_G \le \ugd G$ and Kropholler and Mislin have conjectured that if $n_G < \infty$ then $\ugd G < \infty$ \cite[Conjecture 43.1]{GuidoBookOfConjectures}, they verified the conjecture for groups of type $\FP_\infty$ \cite{KrophollerMislin-GroupsOnFinDimSpacesWithFinStab} and later L\"uck proved it for groups with a bound on the lengths of chains of finite subgroups \cite{Luck-TypeOfTheClassifyingSpace}.  

The algebraic invariant best suited to the study of $\ugd G$ is the Bredon cohomological dimension.  Bredon cohomology was introduced in \cite{Bredon-EquivariantCohomologyTheories}, and extended to infinite groups in \cite{Lueck}.  We denote by $\ucd G $ the Bredon cohomological dimension over $R$.

The $\mathfrak{F}$-cohomology was suggested by Nucinkis as an algebraic analog of $n_G$ \cite{Nucinkis-CohomologyRelativeGSet}, it is a special case of the relative homology of Mac Lane \cite{MacLane-Homology} and Eilenberg--Moore \cite{EilenbergMoore-FoundationsOfRelativeHomologicalAlgebra}.  
Let $\mathfrak{F}$ denote the family of finite subgroups of $G$ and let $\Delta$ denote the $G$-set $\coprod_{H \in \mathfrak{F}} G/H$, we say that a module is \emph{$\mathfrak{F}$-projective} if it is a direct summand of a module of the form $N \otimes_{R} R \Delta$ where $N$ is any $R G$-module.  Short exact sequences are replaced with $\mathfrak{F}$-split short exact sequences---short exact sequences which split when restricted to any finite subgroup of $G$, or equivalently which split when tensored with $R \Delta$.  The class of $\mathfrak{F}$-split short exact sequences is allowable in the sense of Mac Lane, and the projective modules with respect to these sequences are exactly the $\mathfrak{F}$-projectives.
There are enough $\mathfrak{F}$-projectives and one can define a cohomology theory, denoted $\FExt^*_{RG}$:
\[  \FExt^*_{RG}(M, N ) = H^*\Hom_{RG}(P_*, N)  \]
Where $P_*$ is a $\mathfrak{F}$-split resolution of $M$ by $\mathfrak{F}$-projective modules.  We define
\[ \FH^n(G, M) = \FExt^*_{RG}(R, M) \]
The $\mathfrak{F}$-cohomological dimension, denoted $\Fcd G$, is the shortest length of a $\mathfrak{F}$-split $\mathfrak{F}$-projective resolution of $R$.

By a result of Bouc and Kropholler--Wall $\Fcd G \le n_G$ \cite{Bouc-LeComplexeDeChainesDunGComplexe,KrophollerWall-GroupActionsOnAlgebraicCellComplexes} but it is unknown if $\Fcd G < \infty$ implies $n_G < \infty$ or if there exist examples where the invariants differ.  Nucinkis posed an algebraic version of the Kropholler--Mislin conjecture, asking if the finiteness of $\Fcd G $ and $\ucd G$ are equivalent \cite{Nucinkis-EasyAlgebraicCharacterisationOfUniversalProperGSpaces}.

A module is \emph{Gorenstein projective} if it is a cokernel in a strong complete resolution of $RG$-modules, these were first defined over an arbitrary ring by Enochs and Jenda \cite{EnochsJenda-GorensteinInjectiveAndProjectiveModules}.  We will give a full explanation of complete resolutions in Section \ref{subsection:complete resolutions}.  The Gorenstein projective dimension $\Gpd M$ is the minimal length of a resolution of $M$ by Gorenstein projective modules.  Equivalently, $\Gpd M \le n$ if and only if $M$ admits a complete resolution of coincidence index $n$ \cite[p.864]{BahlekehDembegiotiTalelli-GorensteinDimensionAndProperActions}.

The \emph{Gorenstein cohomological dimension} of a group $G$, denoted $\Gcd G$, is the Gorenstein projective dimension of $R$.  If $G$ is virtually torsion-free then $\Gcd G = \vcd G$ \cite[Remark 2.9(1)]{BahlekehDembegiotiTalelli-GorensteinDimensionAndProperActions}.  Indeed the Gorenstein cohomology can be seen of as a generalisation of the virtual cohomological dimension.  Bahlekeh, Dembegioti and Talelli have conjectured that $\Gcd G < \infty$ implies that $\ucd G < \infty$ \cite[Conjecture 3.5]{BahlekehDembegiotiTalelli-GorensteinDimensionAndProperActions}.

By \cite[Lemma 2.21]{AsadollahiBahlekehSalarian-HierachyCohomologicalDimensions}, every permutation $R G$-module with finite stabilisers is Gorenstein projective, so combining with \cite[Lemma 3.4]{Gandini-CohomologicalInvariants} gives that $\Gcd G \le \Fcd G$.  

In general we have the following chain of inequalities.
\[ \Gcd G \le \Fcd G \le n_G \le \ucd G \]
We prove the following:

\theoremstyle{plain}\newtheorem*{CustomThmA*}{Theorem \ref{thm:Fcd finite then Fcd=Gcd}}
\begin{CustomThmA*}
If $\Fcd G < \infty$ then $\Fcd G = \Gcd G$.
\end{CustomThmA*}

We don't know if $\Gcd G < \infty$ implies $\Fcd G < \infty$, although if $\Gcd G = 0$ or $1$ then $\Gcd G = \Fcd G = \ucd G$ \cite[Proposition 2.19]{AsadollahiBahlekehSalarian-HierachyCohomologicalDimensions}\cite[Theorem 3.6]{BahlekehDembegiotiTalelli-GorensteinDimensionAndProperActions}.  Additionally if $G$ is in Kropholler's class $\HF$ and has a bound on the orders of its finite subgroups then $\Fcd G = \Gcd G$ (see Example \ref{example:HF}).

Generalising a construction of Avramov--Martsinkovsky, Asadollahi--Bahlekeh--Salarian showed that if $\Gcd G < \infty$ then there is a long exact sequence of cohomological functors relating the group cohomology, the complete cohomology and the Gorenstein cohomology \cite{AvramovMartsinkovsky-AbsoluteRelativeAndTateCohomology-FiniteGorenstein,AsadollahiBahlekehSalarian-HierachyCohomologicalDimensions}.  Our result follows from constructing a similar long exact sequence relating the $\mathfrak{F}$-cohomology, the complete $\mathfrak{F}$-cohomology (defined in Section \ref{subsection:complete F cohomology}) and a new cohomology theory we call the $\mathfrak{F}_G$-cohomology defined in Section \ref{section:FGcohomology}.  These two long exact sequences fit into a commutative diagram, see Proposition \ref{prop:AM seqs diagram}.  It appears that the requirement in Theorem \ref{thm:Fcd finite then Fcd=Gcd} that $\Fcd G < \infty$ will be difficult to circumvent since this new long exact sequence cannot be constructed for all groups.

In Section \ref{section:Group Extensions} we use that the Gorenstein cohomological dimension is subadditive to improve upon a result of Degrijse on the behaviour of $\Fcd$ under group extensions \cite[Theorem B]{Degrijse-ProperActionsAndMackeyFunctors}.  Degrijse phrased his result in terms of Bredon cohomological dimension of $G$ with coefficients restricted to cohomological Mackey functors, but this invariant is equal to $\Fcd G$ \cite[Theorem 6.2]{Me-CohomologicalMackey}.

\theoremstyle{plain}\newtheorem*{CustomThmC*}{Corollary \ref{cor:Fcd subadditive if finite}}
\begin{CustomThmC*}
Given a short exact sequence of groups
\[ 1 \longrightarrow N \longrightarrow G \longrightarrow Q \longrightarrow 1  \]
if $\Fcd G < \infty$ then $ \Fcd G \le \Fcd N + \Fcd Q $.
\end{CustomThmC*}

In Section \ref{section:cdQ} we use the Avramov--Martsinkovsky long exact sequence to prove the following.
\theoremstyle{plain}\newtheorem*{CustomThmD*}{Proposition \ref{prop:cdQG finite then cdQG le GcdG}}
\begin{CustomThmD*}
 If $\Gcd G < \infty$ and $\cd_{\QQ} G < \infty$ then $\cd_{\QQ} G \le \Gcd G$.
\end{CustomThmD*}

\begin{Acknowledgements*}
 The author would like to thank his supervisor Brita Nucinkis for all her advice and support.
\end{Acknowledgements*}

\section{Preliminaries}\label{section:preliminaries}

\subsection{Complete Resolutions and Complete Cohomology}\label{subsection:complete resolutions}

A \emph{weak complete resolution} of a module $M$ is an acyclic resolution $T_*$ of projective modules which coincides with an ordinary projective resolution $P_*$ of $M$ in sufficiently high degree.  The degree in which the two coincide is called the \emph{coincidence index}.  A weak complete resolution is called a \emph{strong complete resolution} if $\Hom_{RG} (T_*, Q)$ is acyclic for every projective module $Q$.  We avoid the term ``complete resolution'' since some authors use it to refer to a weak complete resolution and others to a strong complete resolution.  

\begin{Prop}\cite[Proposition 2.8]{AsadollahiBahlekehSalarian-HierachyCohomologicalDimensions}
 A group $G$ admits a strong complete resolution if and only if $\Gcd G < \infty$ 
\end{Prop}

The advantage of strong complete resolutions is that given strong complete resolutions $T_*$ and $S_*$ of module $M$ and $N$, any module homomorphism $M \to N$ lifts to a morphism of strong complete resolutions $T_* \to S_*$ \cite[Lemma 2.4]{CornickKropholler-OnCompleteResolutions}.  Thus they can be used to define a cohomology theory: given a strong complete resolution $T_*$ of $M$ we define $$\widehat{\Ext}^*_{RG}(M,-) \cong H^*\Hom_{RG}(T_*, -)$$
We also set $\widehat{H}^*(G, -) = \widehat{\Ext}^*_{RG}(R, -)$.   This coincides with the complete cohomology of Mislin \cite{Mislin-TateViaSatellites}, Vogel \cite{Goichot-HomologieDeTateVogel}, and Benson--Carlson \cite{BensonCarlson-ProductsInNegativeCohomology} (see \cite[Theorem 1.2]{CornickKropholler-OnCompleteResolutions} for a proof).  Recall that the complete cohomology is itself a generalisation of the Farrell--Tate Cohomology, defined only for groups with finite virtual cohomological dimension \cite[\S X]{Brown}.

Even weak complete resolutions do not always exist, for example a free Abelian group of infinite rank cannot admit a weak complete resolution \cite[Corollary 2.10]{MislinTalelli-FreelyProperlyOnFDHomotopySpheres}.  It is conjectured by Dembegioti and Talelli that a $\ZZ G$-module admits a weak complete resolution if and only if it admits a strong complete resolution \cite[Conjecture B]{DembegiotiTalelli-ANoteOnCompleteResolutions}.

\subsection{\texorpdfstring{$\mathfrak{F}$}{F}-Cohomology}
This section contains two technical lemmas we will need later.  

If $M$ is any $RG$-module and $F_i = R \Delta^i$ is the standard $\mathfrak{F}$-split resolution of $R$ \cite[p.342]{Nucinkis-EasyAlgebraicCharacterisationOfUniversalProperGSpaces}, then $F_* \otimes_{R} M$ is an $\mathfrak{F}$-split $\mathfrak{F}$-projective resolution of $M$.  Thus we've shown:

\begin{Lemma}\label{lemma:Fsplit resolutions exist}
 $\mathfrak{F}$-split $\mathfrak{F}$-projective resolutions exist for all $RG$-modules $M$.
\end{Lemma}

There is also a version of the Horseshoe Lemma, proved as in \cite[Lemma 8.2.1]{EnochsJenda-RelativeHomologicalAlgebra1}.
\begin{Lemma}[Horseshoe Lemma]\label{lemma:Fcohomology horseshoe}
 If 
\[0 \longrightarrow A \longrightarrow B \longrightarrow C \longrightarrow 0 \]
is an $\mathfrak{F}$-split short exact sequence and $P_*$ and $Q_*$ are $\mathfrak{F}$-split $\mathfrak{F}$-projective resolutions of $A$ and $C$ respectively then there is an $\mathfrak{F}$-split $\mathfrak{F}$-projective resolution $S_*$ of $B$ such that $S_i = P_i \oplus Q_i$ and there is an $\mathfrak{F}$-split short exact sequence of augmented complexes 
\[0 \longrightarrow \tilde P_* \longrightarrow \tilde S_* \longrightarrow \tilde Q_* \longrightarrow 0 \]
\end{Lemma}

\subsection{Complete \texorpdfstring{$\mathfrak{F}$}{F}-cohomology}\label{subsection:complete F cohomology}

Nucinkis constructs a complete $\mathfrak{F}$-cohomology in \cite{Nucinkis-CohomologyRelativeGSet}, we give a brief outline here.  An \emph{$\mathfrak{F}$-complete resolution} $T_*$ of $M$ is an acyclic $\mathfrak{F}$-split complex of $\mathfrak{F}$-projectives which coincides with an $\mathfrak{F}$-split $\mathfrak{F}$-projective resolution of $M$ in high enough dimensions.  
An \emph{$\mathfrak{F}$-strong} $\mathfrak{F}$-complete resolution $T_*$ has $\Hom_{RG}(T_*, Q)$ exact for all $\mathfrak{F}$-projectives $Q$.  
Given such a $T_*$ we define
\[ \widehat{\mathfrak{F}\negthinspace\Ext}_{RG}^*(M,-) = H^* \Hom_{RG}(T_*, -) \]
\[ \FhatH^*(G, -) = \widehat{\mathfrak{F}\negthinspace\Ext}_{RG}^*(R,-) \]
Nucinkis also describes a Mislin style construction and a Benson--Carlson construction of complete $\mathfrak{F}$-cohomology defined for all groups, proves they are equivalent, and proves that whenever there exists an $\mathfrak{F}$-complete resolution they agree with the definition above.

\subsection{Gorenstein Cohomology}\label{subsection:Gorenstein cohomology}
The Gorenstein cohomology is, like the $\mathfrak{F}$-cohomology, a special case of the relative homology of Mac Lane \cite{MacLane-Homology} and Eilenberg--Moore \cite{EilenbergMoore-FoundationsOfRelativeHomologicalAlgebra}.

Recall that a module is Gorenstein projective if it is a cokernel in a strong complete resolution.  
An acyclic complex $C_*$ of Gorenstein projective modules is \emph{G-proper} if $\Hom_{R G}(Q, C_*)$ is exact for every Gorenstein projective $Q$.  The class of G-proper short exact sequences is allowable in the sense of Mac Lane \cite[\S IX.4]{MacLane-Homology}.  The projectives objects with respect to G-proper short exact sequences are exactly the Gorenstein projectives.  For $M$ and $N$ any $RG$-modules, we define
\[ \GExt_{RG}^*(M, N) = H^*\Hom_{RG} (P_*, N) \]
\[ \GH^*(G, N) = \GExt_{RG}^*(R, N) \]
Where $P_*$ is a G-proper resolution of $M$ by Gorenstein projectives.  

The usual method of producing a ``Gorenstein projective dimension'' of a module $M$ in this setting would be to look at the shortest length of a G-proper resolution of $M$ by Gorenstein projectives.  A priori this could be larger than the Gorenstein projective dimension defined in the introduction, where the G-proper condition is not required.  Fortunately there is the following theorem of Holm:

\begin{Theorem}\label{theorem:Holms theorem on proper vs non proper res}\cite[Theorem 2.10]{Holm-GorensteinHomologicalDimensions}
 If $M$ has finite Gorenstein projective dimension then $M$ admits a G-proper Gorenstein projective resolution of length $\Gpd M$.
\end{Theorem}

Generalising an argument of Avramov and Martsinkovsky in \cite{AvramovMartsinkovsky-AbsoluteRelativeAndTateCohomology-FiniteGorenstein}, Asadollahi Bahlekeh and Salarian construct a long exact sequence:

\begin{Theorem}[Avramov--Martsinkovsky long exact sequence]\label{theorem:AM LES}\cite[Theorem 3.11]{AsadollahiBahlekehSalarian-HierachyCohomologicalDimensions}
For a group $G$ with $\Gcd G < \infty$, there is a long exact sequence of cohomology functors
\[ 0 \longrightarrow \GH^1(G, -) \longrightarrow H^1(G, -) \longrightarrow \cdots \]
\[ \cdots \longrightarrow \GH^n(G, -) \longrightarrow H^n(G, -) \longrightarrow \widehat{H}^n(G, -) \longrightarrow \GH^{n+1}(G, -) \longrightarrow \cdots\]
\end{Theorem}

The construction relies on the complete cohomology being calculable via a complete resolution, hence the requirement that $\Gcd G < \infty$.

We will need the following lemma later:

\begin{Lemma}\label{lemma:Gproper res of R is Fsplit}
 Any G-proper resolution of $R$ is $\mathfrak{F}$-split.
\end{Lemma}
\begin{proof}
 If $P_*$ is a G-proper resolution of $R$ then since $R[G/H]$ is a Gorenstein projective \cite[Lemma 2.21]{AsadollahiBahlekehSalarian-HierachyCohomologicalDimensions},
\[ \Hom_{RG}(R[G/H], P_*) \cong \Hom_{RH}(R, P_*) \cong P_*^H \]
 is exact, thus $P_*$ is $\mathfrak{F}$-split \cite[Remark 5.5, Lemma 5.11]{Me-CohomologicalMackey}.
\end{proof}

\section{\texorpdfstring{$\mathfrak{F}_G$}{F\_G}-cohomology}\label{section:FGcohomology}

\subsection{Construction}

We define another special case of relative homology, which we call the $\mathfrak{F}_G$-cohomology.  It enables us to build an Avramov--Martsinkovsky long exact sequence of homological functors containing $\FH^*$ and $\FhatH^*$.  

We define an \emph{$\mathfrak{F}_G$-projective} to be the cokernel in a $\mathfrak{F}$-complete $\mathfrak{F}$-strong resolution and say a complex $C_*$ of $RG$-modules is \emph{$\mathfrak{F}_G$-proper} if $\Hom_{RG}(Q, C_*)$ is exact for any $\mathfrak{F}_G$-projective $Q$.  The $\mathfrak{F}_G$-proper short exact sequences form an allowable class in the sense of Mac Lane, whose projective objects are the $\mathfrak{F}_G$-projectives --- to check the class of $\mathfrak{F}_G$-proper short exact sequences is allowable we need only check that given a $\mathfrak{F}_G$-proper short exact sequence, any isomorphic short exact sequence is $\mathfrak{F}_G$-proper and that for any $RG$-module $A$ the short exact sequences
\[ 0 \longrightarrow A \stackrel{\id}{\longrightarrow} A \longrightarrow 0 \longrightarrow 0 \]
and
\[ 0 \longrightarrow 0 \longrightarrow A \stackrel{\id}{\longrightarrow} A \longrightarrow 0 \]
are $\mathfrak{F}_G$-proper.

We don't know if the class of $\mathfrak{F}_G$-projectives is precovering (see \cite[\S 8]{EnochsJenda-RelativeHomologicalAlgebra1}), so we don't know if there always exists an $\mathfrak{F}_G$-proper $\mathfrak{F}_G$-projective resolution.  However if $A$ and $B$ admit $\mathfrak{F}_G$-proper $\mathfrak{F}_G$-resolutions $P_*$ and $Q_*$ respectively then any map $A \longrightarrow B$ induces a map of resolutions $P_* \longrightarrow Q_*$ which is unique up to chain homotopy equivalence \cite[IX.4.3]{MacLane-Homology} and we have a slightly weaker form of the Horseshoe Lemma, the proof of which is as in \cite[8.2.1]{EnochsJenda-RelativeHomologicalAlgebra1}:

\begin{Lemma}[Horseshoe Lemma]\label{lemma:weak horseshoe for FG}
 Suppose 
 \[ 0 \longrightarrow A \longrightarrow B \longrightarrow C \longrightarrow 0 \]
 is a $\mathfrak{F}_G$-proper short exact sequence of $RG$-modules and both $A$ and $C$ admit $\mathfrak{F}_G$-proper $\mathfrak{F}_G$-projective resolutions $P_*$ and $Q_*$ then there is an $\mathfrak{F}_G$-proper resolution $S_*$ of $B$ such that $S_i = P_i \oplus Q_i$ and there is an $\mathfrak{F}_G$-proper short exact sequence of augmented complexes 
\[0 \longrightarrow \tilde P_* \longrightarrow \tilde S_* \longrightarrow \tilde Q_* \longrightarrow 0 \]
\end{Lemma}

For any module $M$ which admits an $\mathfrak{F}_G$-proper resolution $P_*$ by $\mathfrak{F}_G$-projectives we define
\[\FGExt^*_{RG}(M, N) = H^*\Hom_{RG}(P_*, N)\]
We define also
\[\FGH^*(G, -) = \FGExt^*_{RG} (R, -) \]

The next lemma follows from Lemma \ref{lemma:weak horseshoe for FG}, see \cite[8.2.3]{EnochsJenda-RelativeHomologicalAlgebra1}.
\begin{Lemma}\label{lemma:long exact sequences in FGExt}
Suppose 
 \[ 0 \longrightarrow A \longrightarrow B \longrightarrow C \longrightarrow 0 \]
 is a $\mathfrak{F}_G$-proper short exact sequence of $RG$-modules and both $A$ and $C$ admit $\mathfrak{F}_G$-proper $\mathfrak{F}_G$-projective resolutions, then there is an $\FGExt^*_{RG}(-,M)$ long exact sequence for any $RG$-module $M$.
\end{Lemma}

For any $RG$-module $M$ the \emph{$\mathfrak{F}_G$ projective dimension} of $G$ denoted $\FGpd M$ is the minimal length of an $\mathfrak{F}_G$-proper resolution of $M$ by $\mathfrak{F}_G$-projectives.  We set $\FGcd G = \FGpd R$.  Note that these finiteness conditions will not be defined unless $R$ admits an $\mathfrak{F}_G$-proper resolution by $\mathfrak{F}_G$-projectives.

One could think of $\mathfrak{F}_G$-cohomology as the ``Gorenstein cohomology relative $\mathfrak{F}$''.
\subsection{Technical Results}

We need a couple of results for the $\mathfrak{F}_G$-cohomology whose analogs are well known for Gorenstein cohomology \cite{Holm-GorensteinHomologicalDimensions}.

We say an $RG$-module $M$ admits a \emph{right} resolution by $\mathfrak{F}$-projectives if there exists an exact chain complex
\[ 0 \longrightarrow M \longrightarrow T_{-1} \longrightarrow T_{-2} \longrightarrow \cdots \]
where the $T_i$ are $\mathfrak{F}$-projectives.  $\mathfrak{F}$-strong right resolutions and $\mathfrak{F}$-split right resolutions are defined as for any chain complex.

\begin{Lemma}\label{lemma:FGprog iff Ext vanishes and strong right res}
 An $RG$-module $M$ is $\mathfrak{F}_G$-projective if and only if $M$ satisfies
\begin{equation*}
 \FExt^i_{RG}(M, Q) \cong 0 \text{ for all $\mathfrak{F}$-projective $Q$ } \tag{$\star$} 
\end{equation*}
for all $i \ge 1$ and $M$ admits a right $\mathfrak{F}$-strong $\mathfrak{F}$-split resolution by $\mathfrak{F}$-projectives.
\end{Lemma}
\begin{proof}
 If $M$ is the cokernel of a $\mathfrak{F}$-strong $\mathfrak{F}$-complete resolution $T_*$ then for all $i \ge 1$ and any $\mathfrak{F}$-projective $Q$,
\[ \FExt^i_{RG}( M, Q ) \cong H^i\Hom_{RG}(T_*^+, Q) \]
Where $T_*^+$ denotes the resolution $T_i^+ = T_i$ if $i \ge 0$ and $T_i^+ = 0$ for $i < 0$.  Then $(\star)$ follows because $T_*$ is $\mathfrak{F}$-strong.  

Conversely given $(\star)$ and an $\mathfrak{F}$-strong right resolution $T_*^-$ then let $T_*^+$ be the standard $\mathfrak{F}$-split resolution for $M$ (Lemma \ref{lemma:Fsplit resolutions exist}), $(\star)$ ensures that $T_*^+$ is $\mathfrak{F}$-strong and splicing together $T_*^+$ and $T_*^-$ gives the required resolution.
\end{proof}

\begin{Lemma}\label{lemma:FExt(Gproj, finite Fpd) = 0}
 If $\Fpd N < \infty$ and $M$ is $\mathfrak{F}_G$-projective then $\FExt^i_{RG}(M, N) = 0$ for all $ i \ge 1$.
\end{Lemma}
\begin{proof}
Let $P_* \longrightarrow N$ be a $\mathfrak{F}$-split $\mathfrak{F}$-projective resolution then by a standard dimension shifting argument 
\[  \FExt^i(M, N) \cong \FExt^{i+j}(M, K_j)  \]
where $K_j$ is the $j^\text{th}$ syzygy of $P_*$.  Since $K_j$ is projective for $j \ge n$ the result follows from Lemma \ref{lemma:FGprog iff Ext vanishes and strong right res}.
\end{proof}

\begin{Prop}\label{prop:syzygy Fproj then FGproper}
Let $A$ be any $RG$-module and $P_* \longrightarrow A$ a length $n$ $\mathfrak{F}$-split resolution of $A$ with $P_i$ $\mathfrak{F}$-projective for $i \ge 1$, then $P_*$ is $\mathfrak{F}_G$-proper.
\end{Prop}
\begin{proof}
The case $n= 0$ is obvious.  If $n = 1$ then for any $\mathfrak{F}_G$-projective $Q$, there is a long exact sequence
\[ 0 \longrightarrow \Hom_{RG}(Q, P_1) \longrightarrow \Hom_{RG}(Q, P_0) \longrightarrow \Hom_{RG}(Q, A) \] 
\[ \longrightarrow \FExt^1_{RG}(Q, P_1) \longrightarrow \cdots \] 
But $\FExt^1_{RG}(Q, P_1) = 0$ by Lemma \ref{lemma:FExt(Gproj, finite Fpd) = 0}.  

Assume $n \ge 2$ and let $K_*$ be the syzygies of $P_*$, then there is an $\mathfrak{F}$-split resolution
\[ 0 \longrightarrow P_n \longrightarrow \cdots \longrightarrow P_{i+1} \longrightarrow K_i \longrightarrow 0 \] 
so $\Fpd K_i < \infty$ for all $i \ge 0$.  Thus every short exact sequence
\[ 0 \longrightarrow K_i \longrightarrow P_i \longrightarrow K_{i-1} \]
is $\mathfrak{F}_G$-proper by Lemma \ref{lemma:FExt(Gproj, finite Fpd) = 0}, so $P_*$ is $\mathfrak{F}_G$-proper.
\end{proof}

\begin{Lemma}[Comparison Lemma]\label{lemma:GF comparison lemma}
 Let $A$ and $B$ be two $RG$-modules with $\mathfrak{F}$-strong $\mathfrak{F}$-split right resolutions by $\mathfrak{F}$-projectives called $S^*$ and $T^*$ respectively, then any map $f: A \to B$ lifts to a map $f_*$ of complexes as shown below:
\[
\xymatrix{
0 \ar[r] & A \ar[r] \ar^f[d] & S^{1} \ar[r] \ar^{f_1}[d] & S^2 \ar[r] \ar^{f_2}[d] & \cdots \\
0 \ar[r] & B \ar[r] & T^{1} \ar[r]  & T^{2} \ar[r] & \cdots 
}
\]
The map of complexes is unique up to chain homotopy and if $f$ is $\mathfrak{F}$-split then so is $f_*$.
\end{Lemma}
\begin{proof}
 The Lemma without the $\mathfrak{F}$-splitting comes from dualising \cite[p.169]{EnochsJenda-RelativeHomologicalAlgebra1}, see also \cite[Proposition 1.8]{Holm-GorensteinHomologicalDimensions}. 

 Assume $f$ is $\mathfrak{F}$-split and consider the map of complexes restricted to $R H$ for some finite subgroup $H$ of $G$.  Let $\iota_*^T$ and $\iota_*^S$ denote the splittings of the top and bottom rows and $s_*$ the splitting of $f_*$, constructed only up to degree $i-1$.  The base case of the induction, when $i = 0$, holds because $f$ is $\mathfrak{F}$-split.
\[
\xymatrix@+15pt{
\cdots \ar^{\partial_{i-2}^S}@/^/[r] & S^{i-1} \ar^{\iota_{i-2}^S}@/^/[l] \ar^{\partial_{i-1}^S}@/^/[r] \ar^{f_{i-1}}@/^/[d] & S^i \ar^{\iota_{i-1}^S}@/^/[l] \ar^{\partial_{i}^S}@/^/[r] \ar^{f_{i}}@/^/[d] & \ar^{\iota_{i}^S}@/^/[l] \cdots \\
\cdots \ar^{\partial_{i-2}^T}@/^/[r] & T^{i-1} \ar^{\iota_{i-2}^T}@/^/[l] \ar^{\partial_{i-1}^T}@/^/[r] \ar^{s_{i-1}}@/^/[u] & T^i \ar^{\iota_{i-1}^T}@/^/[l] \ar^{\partial_{i}^T}@/^/[r] & \ar^{\iota_{i}^T}@/^/[l] \cdots 
}
\]
Let $s_{i} = \partial_{i-1}^S \circ s_{i-1} \circ \iota_{i-1}^T$.  Then
\begin{align*}
 f_i \circ s_i &= f_i \circ \partial_{i-1}^S \circ s_{i-1} \circ \iota_{i-1}^T \\
 &= \partial_{i-1}^T \circ f_{i-1} \circ s_{i-1} \circ \iota^T_{i-1} \\
 &= \partial_{i-1}^T \circ \iota_{i-1}^T \\
 &= \id_{T^i}
\end{align*}
Where the second equality is the commutativity condition coming from the fact that $f_*$ is a chain map.
\end{proof}

\subsection{An Avramov--Martsinkovsky Long Exact Sequence in \texorpdfstring{$\mathfrak{F}$}{F}-cohomology}

\begin{Theorem}\label{theorem:FAM LES}
Given an $\mathfrak{F}$-strong $\mathfrak{F}$-complete resolution of $R$ there is a long exact sequence
\[ 0 \longrightarrow \mathfrak{F}_G H^1(G, -) \longrightarrow \cdots  \]
\[ \cdots \longrightarrow \FhatH^{n-1}(G, -) \longrightarrow \FGH^n(G, -) \longrightarrow \FH^n(G, -)  \]
\[ \longrightarrow \FhatH^n(G, -) \longrightarrow \FGH^{n+1}(G, -) \longrightarrow \cdots \]
\end{Theorem}
\begin{proof}
 We follow the proof in \cite[\S 3]{AsadollahiBahlekehSalarian-HierachyCohomologicalDimensions}.  Consider an $\mathfrak{F}$-strong $\mathfrak{F}$-complete resolution $T_*$ coinciding with an $\mathfrak{F}$-projective $\mathfrak{F}$-split resolution $P_*$ in sufficiently high dimension.  We may choose $\theta_* : T_* \longrightarrow P_*$ to be $\mathfrak{F}$-split by Lemma \ref{lemma:GF comparison lemma} and without loss of generality we may also assume that $\theta_i$ is surjective for all $i$.

 Truncating at position $0$ and adding cokernels gives the bottom two rows of the diagram below, the row above is the row of kernels.  Note that the map $A \to R$ is necessarily surjective since the maps $T_0 \to P_0$ and $P_0 \to R$ are surjective.

\[ 
\xymatrix{
\cdots \ar[r] & 0   \ar[r] \ar[d] & K_{n-1} \ar[r] \ar[d] & \cdots \ar[r] & K_0 \ar[r] \ar[d] & K   \ar[r] \ar[d] & 0 \\
\cdots \ar[r] & T_n \ar[r] \ar[d] & T_{n-1} \ar[r] \ar[d] & \cdots \ar[r] & T_0 \ar[r] \ar[d] & A   \ar[r] \ar[d] & 0 \\
\cdots \ar[r] & P_n \ar[r]        & P_{n-1} \ar[r]        & \cdots \ar[r] & P_0 \ar[r]        & R \ar[r]        & 0
}
\]
We make some observations about the diagram:  Firstly since the module $A$ is the cokernel of a $\mathfrak{F}$-strong $\mathfrak{F}$-complete resolution, $A$ is $\mathfrak{F}_G$ projective.  Secondly in degree $i \ge 0$ the columns are $\mathfrak{F}$-split and the $P_i$ are $\mathfrak{F}$-projective, thus the $K_i$ are $\mathfrak{F}$-projective for all $i \ge 0$.  Thirdly the far right vertical short exact sequence is $\mathfrak{F}$-split since the degree $0$ column and the rows are $\mathfrak{F}$-split. Finally the top row is exact and $\mathfrak{F}$-split since the other two rows are.  

Apply the functor $\Hom_{RG} (-, M)$ for an arbitrary $R G$-module $M$ and take homology.  This gives a long exact sequence
\[ \cdots \longrightarrow  \FH^i (G, M) \longrightarrow \FhatH^i (G, M) \longrightarrow H^i\Hom_{RG}(K_*, M) \longrightarrow \cdots \]
We can simplify the right hand term:
\begin{align*}
H^i\Hom_{RG}(K_*, M) &\cong \FGExt^i_{RG}(K, M) \\
&\cong \FGH^{i+1}(G, M)
\end{align*}
 Where the first isomorphism is because, by Proposition \ref{prop:syzygy Fproj then FGproper}, the top row is $\mathfrak{F}_G$-proper.  For the second isomorphism note that the short exact sequence
\[ 0 \longrightarrow K \longrightarrow A \longrightarrow R \longrightarrow 0 \]
is $\mathfrak{F}_G$-proper by Proposition \ref{prop:syzygy Fproj then FGproper}, so 
\[ 0 \longrightarrow K_{n-1} \longrightarrow \cdots \longrightarrow K_0 \longrightarrow A \longrightarrow R \longrightarrow 0 \]
is an $\mathfrak{F}_G$-proper $\mathfrak{F}_G$-projective resolution of $R$.  Thus the second isomorphism follows from the short exact sequence and Lemma \ref{lemma:long exact sequences in FGExt}.
\end{proof}

\begin{Cor}
 If $R$ admits an $\mathfrak{F}$-strong $\mathfrak{F}$-complete resolution then $\FGcd G < \infty$.
\end{Cor}
\begin{proof}
 In the proof of the theorem we assumed an $\mathfrak{F}$-strong $\mathfrak{F}$-complete resolution of $R$ and built a finite length $\mathfrak{F}_G$-proper resolution of $R$ by $\mathfrak{F}_G$-projectives.
\end{proof}

\begin{Prop}\label{prop:AM seqs diagram}
If the Avramov--Martsinkovsky long exact sequence and the long exact sequence of Theorem \ref{theorem:FAM LES} both exist, there is a commutative diagram:
\[
\xymatrix{
 \cdots \ar[r] & \FhatH^{n-1} \ar^{\gamma_{n-1}}[d] \ar[r] & \FGH^n \ar[r] \ar^{\alpha_n}[d] & \FH^n \ar[r] \ar^{\beta_n}[d] & \FhatH^n \ar^{\gamma_n}[d] \ar[r] & \FGH^{n+1} \ar[r] \ar^{\alpha_{n+1}}[d] & \cdots \\
 \cdots \ar[r] &  \widehat{H}^{n-1} \ar[r] & \GH^n \ar[r] \ar_{\eta_n}[ru] & H^n \ar[r] & \widehat{H}^n \ar[r]   & \GH^{n+1} \ar[r] \ar[ru] & \cdots \\
}
\]
Where for conciseness we have written $H^n$ for $H^n(G, -)$ etc.
\end{Prop}
\begin{proof}
The construction of the Avramov--Martsinkovsky long exact sequence is analogous to the proof of Theorem \ref{theorem:AM LES}, we give a quick sketch below as we will need the notation.  Take a strong complete resolution $T^\prime_*$ of $R$ coinciding with a projective resolution $P_*^\prime$ in high dimensions and let $A^\prime$ be the zeroth cokernel of $T_*^\prime$.  Thus $A^\prime$ is Gorenstein projective.  Again, the map $T_*^\prime \to P_*^\prime$ is assumed surjective and the kernel $K_*^\prime$ is a projective resolution of $K^\prime$, the kernel of the map $A^\prime \longrightarrow R$.  Applying $\Hom_{RG}(-, M)$, for some $RG$-module $M$, to the short exact sequence of complexes 
\[0 \longrightarrow K_* \longrightarrow T_* \longrightarrow P_* \longrightarrow 0  \] 
gives the Avramov--Martsinkovsky long exact sequence.

Let $T_*$, $P_*$, $K_*$, $K$ and $A$ be as defined in the proof of Theorem \ref{theorem:AM LES}.  There is a commutative diagram of chain complexes
\[ 
\xymatrix{
0 \ar[r] & {K}_* \ar[r]  & {T}_* \ar[r] & {P}_* \ar[r] & 0\\
0 \ar[r] & {K}_*^\prime \ar[r] \ar^\alpha[u] & {T}_*^\prime \ar[r] \ar^\gamma[u] & {P}_*^\prime \ar[r] \ar^\beta[u] & 0 
}\]

Where the maps $\beta$ exists by the comparison theorem for projective resolutions and $\gamma$ exists by the comparison theorem for strong complete resolutions \cite[Lemma 2.4]{CornickKropholler-OnCompleteResolutions}.   The map $\alpha$ is the induced map on the kernels.  Applying $\Hom_{RG}(-, M)$ for some $RG$-module $M$, and taking homology, the maps $\alpha$, $\beta$ and $\gamma$ induce the maps $\alpha_*$, $\beta_*$ and $\gamma_*$.

Finally we construct the map $\eta_n: \GH^n(G, -) \longrightarrow \FH^n(G, -)$.  Let $B_*$ be a G-proper Gorenstein projective resolution and recall $P_*$ is an $\mathfrak{F}$-split resolution by $\mathfrak{F}$-projectives.  Then $B_*$ is $\mathfrak{F}$-split (Lemma \ref{lemma:Gproper res of R is Fsplit}) so there is a chain map  $P_* \to B_*$ inducing $\eta_*$ on cohomology.

Commutativity is obvious for the diagram with the maps $\eta_i$ removed, leaving us with two relations to prove.  Let 
$$\varepsilon_n^G : \GH^n(G, -) \longrightarrow H^n(G, -)$$
denote the map from the commutative diagram.  This is the map induced by comparison of a resolution of Gorenstein projectives and ordinary projectives \cite[3.2,3.11]{AsadollahiBahlekehSalarian-HierachyCohomologicalDimensions}.  We get $\beta_* \circ \eta_* = \varepsilon^G_* $, since all the maps are induced by comparison of resolutions, and such maps are unique up to chain homotopy equivalence.

The final commutativity relation, that $\eta_* \circ \alpha_* = \varepsilon_*^{\mathfrak{F}_G} $, is the most difficult to show.  Here 
\[\varepsilon^{\mathfrak{F}_G}_n: \FGH^n(G, -) \longrightarrow \FH^n(G, -)\]
denotes the map from the commutative diagram, it is induced by comparison of resolutions.

Here is a commutative diagram showing the resolutions involved:
\[
\xymatrix@-15pt{
& 0 \ar[rr] & & K \ar[rr] & & A \ar[rr] & & R \ar[rr] & & 0 \\
0 \ar[rr] & & K_* \ar[rr] \ar[ru] & \ar[u] & T_* \ar[ru] \ar[rr] & \ar[u] & P_* \ar[ru] \ar[rr] & \ar[u] & 0 & \\
& 0 \ar@{-}[r] & \ar[r] & K^\prime \ar@{-}[r] \ar@{-}[u] & \ar[r] & A^\prime \ar@{-}[r] \ar@{-}[u] & \ar[r] & R \ar@{-}[u] \ar[rr] & & 0 \\
0 \ar[rr] & & K_*^\prime \ar[rr] \ar[ru] \ar[uu] & & T_*^\prime \ar[rr] \ar[ru] \ar[uu] & & P_*^\prime \ar[rr] \ar[ru] \ar[uu] & & 0
}
\]

Let $L_*$ be the chain complex defined by $L_i = K_{i-1}$ for all $i \ge 1$ and $L_0 = A$, with boundary map at $i =1$ the composition of the maps $K_0 \to K$ and $K \to A$.  Thus $L_*$ is acyclic except at degree zero where $H_0L_* = R$.  Similarly let $L_*^\prime$ denote chain complex with $L_i^\prime = K_{i-1}^\prime$ for all $i \ge 1$ and $L_0^\prime = A^\prime$ augmented by $A^\prime$, so $L_*^\prime$ is acyclic except at degree zero where $H_0L_*^\prime = R$.  Note that $L_*$ is an $\mathfrak{F}_G$-proper resolution of $R$ by Proposition \ref{prop:syzygy Fproj then FGproper} and $L_*^\prime$ is a G-proper resolution of $R$ by the Gorenstein cohomology version of the same proposition.

Recall that the maps $\varepsilon_*^{\mathfrak{F}_G}$ and $\eta_*$ are induced by comparison of resolutions:  $\varepsilon^{\mathfrak{F}_G}_*$ is induced by a map $P_* \to L_*$ and $\eta_*$ is induced by a map $P_* \to L_*^\prime$.  The map 
\[\FGExt^i_{RG}(K, -) \longrightarrow \GExt^i_{RG}(K^\prime, -)\]
is induced by $\alpha: K_*^\prime \longrightarrow K_*$.  Thus the map 
\[\alpha_* : \FGH^n(G, -) \longrightarrow \GH^n(G, -)\]
is induced by $L_*^\prime \longrightarrow L_*$.  The diagram below is the one we must show commutes.  
\[
\xymatrix{
\FGH^n(G, -) \cong H^n\Hom_{RG}(L_* , -) \ar^{\alpha_n}[d] \ar^{\varepsilon_n^{\mathfrak{F}_G}}[r] & \FH^n(G, -) \cong H^n \Hom_{RG}(P_* , -) \\
\GH^n(G, -) \cong H^n\Hom_{RG}(L^\prime_* , -) \ar_{\eta_n}[ur] & \\
}
\]
Since the composition $P_*$ to $L^\prime_*$ to $L_*$ is a map of resolutions from $P_*$ to $L_*$, and such maps are unique up to chain homotopy equivalence, this completes the proof.
\end{proof}
\begin{Cor}
Given an $\mathfrak{F}$-strong $\mathfrak{F}$-complete resolution of $R$, $\Gcd G = n < \infty$ implies $\FH^i(G, -)$ injects into $\FhatH^i(G, -)$ for all $i \ge n+1$.
\end{Cor}
\begin{proof}
$\Gcd G < \infty$ implies the Avramov--Martsinkovsky long exact sequence exists (Theorem \ref{theorem:AM LES}). Consider the the commutative diagram of Proposition \ref{prop:AM seqs diagram}.  The map
\[ \FGH^i(G, -) \longrightarrow \FH^i(G, -) \]
factors as $\eta_i \circ \alpha_i = 0$, so since $\GH^{i}(G, -) = 0$ for all $i \ge n+1$, $\FH^i(G, -)$ injects into $\FhatH^i(G, -)$ for all $i \ge n+1$.  
\end{proof}

\begin{Theorem}\label{thm:Fcd finite then Fcd=Gcd}
 If $\Fcd G < \infty$ then $\Fcd G = \Gcd G$.
\end{Theorem}
\begin{proof}~
 We know already that $\Gcd G \le \Fcd G$ (see Section \ref{section:introduction}).  If $\Fcd G < \infty$ then it is trivially true that $\mathfrak{F}$ admits an $\mathfrak{F}$-strong $\mathfrak{F}$-complete resolution, thus $\FH^i(G, -)$ injects into $\FhatH^i(G, -)$ for all $i \ge \Gcd G + 1$, but $\FhatH^i(G, -)$ is always zero since $\Fcd G < \infty$ \cite[4.1(i)]{Kropholler-OnGroupsOfTypeFP_infty}.
\end{proof}

\begin{Example}\label{example:HF} 
Let $R = \ZZ$ for this example.  Kropholler introduced the class $\HF$ of hierarchically decomposable groups in \cite{Kropholler-OnGroupsOfTypeFP_infty} as the smallest class of groups such that if there exists a finite dimensional contractible $G$-CW complex with stabilisers in $\HF$ then $G \in \HF$.  Let $\HF_b$ denote the subclass of $\HF$ containing groups with a bound on the orders of their finite subgroups.

The $\ZZ G$-module $B(G, \ZZ)$ of bounded functions from $G$ to $\ZZ$ was first studied in \cite{KrophollerTalelli-PropertyOfFundamentalGroupsOfGraphOfFiniteGroups}, Kropholler and Mislin proved that if $G$ is $\HF$ with a bound on lengths of chains of finite subgroups and $\pd_{\ZZ G} B(G, \ZZ ) < \infty$ then $\ucd G < \infty$, in particular $\Fcd G < \infty$ \cite{KrophollerMislin-GroupsOnFinDimSpacesWithFinStab}.  If $\Gcd G < \infty$ then $\pd_{\ZZ G}B(G, \ZZ ) < \infty$ \cite[2.10]{AsadollahiBahlekehSalarian-HierachyCohomologicalDimensions}\cite[Theorem C]{CornickKropholler-HomologicalFinitenessConditions}.  Thus if $G \in \HF_b$ then $\Gcd G = \Fcd G$.
\end{Example}

\section{Group Extensions}\label{section:Group Extensions}

In \cite[Theorem 6.2]{Me-CohomologicalMackey} the author shows that for all groups $\Fcd G =\HeckeFcd G $, where $\HeckeFcd G$ denotes the Bredon cohomological dimension of $G$ with coefficients restricted to cohomological Mackey functors.  The invariant $\HeckeFcd G$ was studied by Degrijse in \cite{Degrijse-ProperActionsAndMackeyFunctors} where he proves the following (though stated for $\HeckeFcd G$ not $\Fcd G$):  
\begin{Theorem}\cite[Theorem B]{Degrijse-ProperActionsAndMackeyFunctors}
 Given a short exact sequence of groups
$$ 1 \longrightarrow N \longrightarrow G \longrightarrow Q \longrightarrow 1  $$
such that every finite index overgroup of $N$ in $G$ has a bound on the orders of the finite subgroups not contained in $N$.  If $\Fcd G < \infty$ then $ \Fcd G \le \Fcd N + \Fcd Q $.
\end{Theorem}

Since Gorenstein cohomological dimension is subadditive under extensions \cite[Remark 2.9(2)]{BahlekehDembegiotiTalelli-GorensteinDimensionAndProperActions}, an application of Theorem \ref{thm:Fcd finite then Fcd=Gcd} removes the condition on the orders of finite subgroups:
\begin{Cor}\label{cor:Fcd subadditive if finite}
 Given a short exact sequence of groups
\[ 1 \longrightarrow N \longrightarrow G \longrightarrow Q \longrightarrow 1  \]
If $\Fcd G < \infty$ then $ \Fcd G \le \Fcd N + \Fcd Q $.
\end{Cor}

\begin{Remark}
Even in the case that $\Fcd Q < \infty$ and $N$ is finite it is unknown if $\Fcd G < \infty$.  However, if it fails in such a case then it necessarily fails when $N$ is a cyclic group of order $p$ \cite[Lemma 6.10]{Me-CohomologicalMackey}. 
\end{Remark}

\section{Rational Cohomological Dimension}\label{section:cdQ}

For this section, let $R = \ZZ$.  Gandini has shown that for groups in $\HF$, $\cd_{\QQ} G \le \Gcd G$ \cite[Remark 4.14]{Gandini-CohomologicalInvariants} and this is the only result we are aware of relating $\cd_{\QQ}G$ and $\Gcd G$.  In Proposition \ref{prop:cdQG finite then cdQG le GcdG} we show that $\cd_{\QQ}G \le \Gcd G$ for all groups with $\cd_{\QQ} G < \infty$.  Recall there are examples of torsion-free groups with $\cd_{\QQ} G < \cd_{\ZZ} G$ \cite[Example 8.5.8]{Davis} and $\Gcd G = \cd_{\ZZ} G$ whenever $\cd_{\ZZ} G < \infty$ \cite[Corollary 2.9]{AsadollahiBahlekehSalarian-HierachyCohomologicalDimensions}, so we cannot hope for equality of $\cd_{\QQ} G$ and $\Gcd G $ in general.

\begin{Question}
 Are there groups $G$ with $\Gcd G < \infty$ but $\cd_{\QQ} G = \infty$?
\end{Question}

\begin{Lemma}
 For any group $G$, $\silp \QQ G \le \silp \ZZ G$.
\end{Lemma}
\begin{proof}
 By \cite[Theorem 4.4]{Emmanouil-OnCertainCohomologicalInvariantsOfGroups}, $\silp \QQ G = \spli \QQ G$ and $\silp \ZZ G = \spli \ZZ G$.  Combining with \cite[Lemma 6.4]{GedrichGruenberg-CompleteCohomologicalFunctors} that $\spli \QQ G \le \spli \ZZ G$ gives the result.
\end{proof}

\begin{Lemma}\label{lemma:complete tensored with Q}
 If $\Gcd G < \infty$ then for any $\QQ G$-module $M$ there is a natural isomorphism
\[ \widehat{H}^*(G, M) \otimes \QQ \cong \widehat{\Ext}^*_{\QQ G}(\QQ,M) \]
\end{Lemma}
\begin{proof}
 Let $T_*$ be a strong complete resolution of $\ZZ$ by $\ZZ G$-modules, then $T_* \otimes \QQ$ is a strong complete resolution of $\QQ$ by $\QQ G$-modules.  By an obvious generalisation of \cite[Lemma 2.2]{MislinTalelli-FreelyProperlyOnFDHomotopySpheres}, if $\silp \QQ G \le \infty$ then any complete $\QQ G$-module resolution is a strong complete $\QQ G$-module resolution, so since $\silp \QQ G < \silp \ZZ G \le \infty$, $T_* \otimes \QQ$ is a strong complete resolution.  This gives a chain of isomorphisms for any $\QQ G$-module $M$:
\begin{align*}
 \widehat{H}^*(G, M) \otimes \QQ &\cong H^* \Hom_{\ZZ G}(T_*, M) \otimes \QQ \\
 &\cong H^* \Hom_{\QQ G}(T_* \otimes \QQ, M)  \\
 &\cong \widehat{\Ext}^*_{\QQ G}(\QQ , M)
\end{align*}
\end{proof}

\begin{Prop}\label{prop:cdQG finite then cdQG le GcdG}
 If $\cd_{\QQ} G < \infty$ then $\cd_{\QQ} G \le \Gcd G$.
\end{Prop}
\begin{proof}
There is nothing to show if $\Gcd G = \infty$ so assume that $\Gcd G < \infty$.  Since $\QQ$ is flat over $\ZZ$, tensoring the Avramov--Martsinkovsky long exact sequence with $\QQ$ preserves exactness.  Combining this with Lemma \ref{lemma:complete tensored with Q} and the well known fact that for any $\QQ G$-module $M$ there is a natural isomorphism \cite[p.2]{Bieri-HomDimOfDiscreteGroups} 
 \[H^*(G, M) \otimes \QQ \cong \Ext^*_{\QQ G}(\QQ, M)\]
 gives the long exact sequence
\[ \cdots \longrightarrow \GH^i(G, M) \otimes \QQ \longrightarrow \Ext^i_{\QQ G}(\QQ, M) \longrightarrow \widehat{\Ext}^i_{\QQ G}(\QQ,M) \longrightarrow \cdots \]
Since $\cd_{\QQ} G < \infty$, we have that $\widehat{\Ext}^i_{\QQ G}(\QQ,M) = 0$ \cite[4.1(i)]{Kropholler-OnGroupsOfTypeFP_infty}.  Thus there is an isomorphism for all $i$,
$$\GH^i(G, M) \otimes \QQ \cong \Ext^i_{\QQ G}(\QQ, M)$$
and the result follows.
\end{proof}

\bibliographystyle{plain}

\end{document}